\documentclass[12pt]{article}

\pdfoutput=1

\usepackage{amsmath}
\usepackage{enumerate}
\usepackage[english]{babel}
\usepackage{amssymb}
\usepackage{a4wide}
\usepackage{graphicx}
\usepackage{color}
\usepackage{xcolor}
\usepackage{makeidx}
\usepackage{enumerate}
\usepackage[utf8]{inputenc}

\newtheorem{theorem}{Theorem}
\newtheorem{lemma}{Lemma}

\newtheorem{example}{Example}

\newtheorem{remark}{Remark}

\newenvironment{proof}{\textbf{Proof.}}{\hspace*{\fill}$\blacksquare$ \\}

\def\ben{\begin{enumerate}}
\def\een{\end{enumerate}}
\def\mxl{\left[\begin{matrix}}
\def\mxr{\end{matrix}\right]}

\begin{document}
\title{Procrustes problem for the inverse eigenvalue problem of normal (skew) $J$-Hamiltonian matrices and normal
$J$-symplectic matrices}

\author{
S. Gigola\footnote{Departamento de Matem\'atica. Facultad de Ingenier\'{\i}a. Universidad de Buenos Aires. Buenos Aires,
Argentina.
E-mail: {\tt sgigola@fi.uba.ar.}}
\hspace{1cm} L. Lebtahi\footnote{Facultat de Ci\`encies Matem\`atiques, Departament de Matem\`atiques,
							Universitat de Val\`encia. Valencia, Spain.
E-mail: {\tt leila.lebtahi@uv.es.} }
\hspace{1cm} N. Thome\footnote{Instituto Universitario de Matem\'atica Multidiscilpinar. Universitat Polit\`ecnica de
					Val\`encia. Valencia, Spain.
E-mail: {\tt njthome@mat.upv.es.}}
}

\date{}

\maketitle

\pagestyle{plain}

\begin{abstract}
A square complex matrix $A$ is called (skew) $J$-Hamiltonian if $AJ$
is (skew) hermitian where $J$ is a real normal matrix such that $J^2=-I$, where $I$ is the identity matrix.
In this paper, we solve the Procrustes problem to find normal (skew) $J$-Hamiltonian solutions for the inverse eigenvalue
problem. In addition, a similar problem is investigated for normal $J$-symplectic matrices.

\end{abstract}



\section{Introduction}\label{s1}

Inverse eigenvalue problems arise as important tools in several research subjects such  as structural design, parameter identification, modeling, etc. \cite{Ch,Jo,Zh}.
Inverse eigenvalue problems consist of
the construction of a matrix $A$ with a determined structure and a specified spectrum.
In the literature, these problems were studied under certain constraints on $A$ as in \cite{Ba,EbMa,GiLeTh1,Na,NaMaNe,ToAbAl,Tr,Tr2,Zh,Zh2}.
For instance, by using hermitian generalized Hamiltonian matrices, the associated inverse eigenvalue problem was solved in \cite{Zh2}.
Moreover, in \cite{Ba}, the general solution of the inverse eigenvalue problem for hermitian generalized skew Hamiltonian matrices
was given and the expression of the solution for the optimal approximation problem was obtained. For symplectic matrices and several factorizations related to them, we refer the reader to \cite{Xu}.

An optimal approximation problem is known as a Procrustes problem.
The name of Procrustes come from Greek mythology, where a bandit made the visitors of his refuge his victims when they went to bed stretching their limbs or cutting them off, fitting them then to the bed. In a recent
interesting book \cite{Ta}, this expression is used to refer the modern society as the back cover says:
"It represents Taleb's view of modern civilization's hubristic side effects--modifying humans to satisfy technology, blaming reality for not fitting economic models, inventing diseases to sell drugs, defining intelligence as what can be tested in a classroom, and convincing people that employment is not slavery.''.

Our interest is to study the Procrustes problem associated to a specific class of matrices
($J$-Hamiltonian, skew $J$-Hamiltonian, and $J$-symplectic matrices  as we will detail in what follows.

It is remarkable that Hamiltonian matrices play an important role in several engineering areas such as optimal quadratic linear
control \cite{Me,Si}, $H_{\infty}$ optimization \cite{ZhDoGl} and the solution of algebraic Riccati equations \cite{Be,La}, among others.
In \cite{MaPe1}, symplectic matrices are studied from a point of view of their connection with predetermined left eigenvalues.

Specifically, for given $k \times k$ real matrices $A,B,$ and $C$ with $B$ and $C$ being symmetric matrices, the algebraic Riccati equation $XA+A^TX-XBX+C=O$ has solution $X=-VU^{-1}$, where the block matrix $\left[\begin{array}{c} U \\ V \end{array}\right]$ spans an $H$-invariant subspace of the corresponding subspace for the Hamiltonian matrix given by
$H=\left[\begin{array}{cc} A & B \\ C & -A^T
 \end{array}\right]$ provided that $U$ is nonsingular \cite{Ku}. We recall that a $2k \times 2k$ real matrix  $G$ is Hamiltonian if the product $G\left[\begin{array}{cc} O & I_k \\ -I_k & O
 \end{array}\right]$ is symmetric, where $I_k$ will stand for the $k \times k$ identity matrix. Moreover, $H$ determines
the form of all the Hamiltonian matrices.

On the other hand, important applications for symplectic matrices arise mainly in Mathematical Physics in areas such as quantum mechanics. For instance, symplectic matrices give the phase-space representation of Gaussian transformations in quantum photonics \cite{ChCe}. In addition, Linear Quadratic Control Problem \cite{Me} requires this kind of matrices in their investigations.

The symbols ${\mathbb C}^{m \times n}$ and ${\mathbb R}^{m \times n}$ stand for the sets of $m \times n$ complex
real matrices.
The symbols $M^*$ and $M^{\dagger}$ will denote the conjugate transpose and the Moore-Penrose inverse of a matrix $M$, respectively.
We remind that for a given rectangular complex matrix $M \in {\mathbb C}^{m \times n}$, its Moore-Penrose inverse is the (unique) matrix $M^{\dagger} \in {\mathbb C}^{n \times m}$
that satisfies $MM^{\dagger}M=M$, $M^{\dagger} M M^{\dagger} = M^{\dagger}$, $(MM^{\dagger})^* = M M^{\dagger}$, and $(M^{\dagger} M)^* = M^{\dagger} M$.
This matrix always exists \cite{BeGr}.
We also need the following notation for both specified orthogonal projectors:
$Q_M=I_n - M^{\dagger}M$ and
$P_M= I_m - M M^{\dagger}$.
It is worth noting that $M^\dag P_{M} = O$ and $P_{M} (M^\dag)^* = O$.
We will use the inner product given by $\langle A,B \rangle={\rm trace}(AB^*)$ for $A,B \in {\mathbb C}^{n \times n}$ and the Frobenius norm of $A$
given by $\|A\|_F = \sqrt{\langle A,A \rangle}$. The symbol $\sigma(A)$ denotes the spectrum of the matrix  $A  \in {\mathbb C}^{n \times n}$.

Let $J \in {\mathbb R}^{n \times n}$ be a normal matrix such that $J^2=-I_n$.
It is well known that a matrix $A\in {\mathbb C}^{n \times n}$ is called $J$-Hamiltonian if $(AJ)^*=AJ$ \cite{GiLeTh2}. Moreover,
$A\in {\mathbb C}^{n \times n}$ is called skew $J$-Hamiltonian if $(AJ)^*=-AJ$.

On the other hand, a matrix $A\in {\mathbb C}^{n \times n}$ is called $J$-symplectic if $A^* J A = J$. Every $J$-symplectic matrix $A$ is nonsingular, with determinant satisfying $|\det(A)| = 1$ and its inverse is given by $A^{-1} = J^{-1} A^* J = -J A^* J$ since $J$ is skew-hermitian. Even more, it is well-known  \cite{Rim,DoJo,MaMa} that $\det(A)=1$,
which means that the algebraic multiplicities of the eigenvalues 1 and $-1$ are both even.
For some applications of symplectic matrices, we refer the reader for example to \cite{MaPe2}.

Throughout all this paper, we will consider a fixed real normal matrix $J$ such that $J^2=-I$,  where $I$ denotes the identity matrix of adequate size.

Observe that, if $A \in {\mathbb C}^{n \times n}$ then the spectral condition for $A$ to be $J$-Hamiltonian is that
 $\lambda \in \sigma(A)$ if and only if $-\bar{\lambda} \in \sigma(A)$, that is, $\lambda \in \sigma(A)$ implies $\{\lambda,-\bar{\lambda}\} \subseteq \sigma(A)$. Even more, when $A \in {\mathbb R}^{n \times n}$, the condition $\lambda \in \sigma(A)$ implies $\{\lambda,-\lambda,\bar{\lambda},-\bar{\lambda}\} \subseteq \sigma(A)$, since in real polynomials non-real roots appear in pairs of the type $a\pm bi$. This fact forces the values of the diagonal
of the given matrix $D$ have to obey these symmetry properties.
Notice that for
$$
A =
\left[\begin{array}{rc}
2i & i \\
-i & 2i
\end{array}
\right]
\qquad \text{ and } \qquad
J =
\left[\begin{array}{cr}
i & 0 \\
0 & -i
\end{array}
\right]
$$
we have that the normal matrix $A$ is $J$-Hamiltonian and $\sigma(A)= \{1+2i,-1+2i\}$.
Thus, $\lambda \in \sigma(A)$ does not imply $\bar{\lambda} \in \sigma(A)$.

Otherwise, if $A \in {\mathbb C}^{n \times n}$ is skew $J$-Hamiltonian then $\lambda \in \sigma(A)$ if and only if
$\bar{\lambda} \in \sigma(A)$, that is, if $\lambda \in \sigma(A)$ then $\{\lambda,\bar{\lambda}\} \subseteq \sigma(A)$.
However, in general, $\lambda \in \sigma(A)$ does not imply $-\lambda \in \sigma(A)$ as the following example shows:
$$
A =
\left[\begin{array}{rc}
3/2 &  i/2
\\
-i/2
& 3/2
\end{array}
\right]
\qquad \text{ and } \qquad
J =
\left[\begin{array}{rc}
0 & 1 \\
-1 & 0
\end{array}
\right].
$$
This skew $J$-Hamiltonian matrix $A$ satisfies $1 \in \sigma(A)=\{1,2\}$ but $-1 \notin \sigma(A)$.

On the other hand, the spectral condition for $A$ to be $J$-symplectic is that
$\lambda \in \sigma(A)$ implies $\displaystyle \left\{\lambda,\bar{\lambda}, \frac{1}{\lambda}, \frac{1}{\bar{\lambda}}\right\} \subseteq \sigma(A)$.

For a given full-column rank matrix $X \in {\mathbb C}^{n \times m}$ and a given diagonal matrix $D \in {\mathbb C}^{m \times m}$, we introduce the set ${\cal S}$ of all solutions of the inverse eigenvalue problem $AX=XD$, that is,
\begin{equation}\label{ecua}
{\cal S} = \{A  \in {\mathbb C}^{n \times n}: AX=XD\}.
\end{equation}
The set ${\cal S}$ when the unknown $A$ is a normal $J$-Hamiltonian matrix was found in \cite{GiLeTh2},
when the unknown $A$ is a normal skew $J$-Hamiltonian matrix, the set ${\cal S}$ was found in \cite{ZhZh}. However, the set ${\cal S}$ for $A$ being a normal $J$-symplectic matrix was not found to our acknowledgement.

We will consider the following sets:
\[
{\cal NJH}:=\{A \in {\mathbb C}^{n \times n}: AA^*=A^*A, (AJ)^*=AJ\},
\]
\[
{\cal NJSH}:=\{A \in {\mathbb C}^{n \times n}: AA^*=A^*A, (AJ)^*=-AJ\},
\]
and
\[
{\cal NJS}:=\{A \in {\mathbb C}^{n \times n}: AA^*=A^*A, A^* J A=J\}.
\]

{\sl Statement problems}:
This paper solves the Procrustes problem associated to normal (respectively, skew) $J$-Hamiltonian matrices, as well as, to normal
$J$-symplectic matrices. That is, 
assuming that ${\cal S} \neq \emptyset$,
for a given matrix $\widetilde{A} \in {\mathbb C}^{n \times n}$, we study the existence and uniqueness of matrices
$\widehat{A} \in {\cal S} \cap {\cal T}$  such that
\begin{equation}\label{Procurstes}
\min_{A \in {\cal S} \cap {\cal T}} \|\widetilde{A} - A\|_F = \|\widetilde{A} - \widehat{A} \|_F,
\end{equation}
for each one of the following cases: ${\cal T} = {\cal NJH}$ or ${\cal T} =  {\cal NJSH}$, or ${\cal T} =  {\cal NJS}$.

This paper is organized as follows.
Section \ref{sec3} solves the Procrustes problem for normal $J$-Hamiltonian matrices.
Section \ref{sec5} is dedicated to solve the Procrustes problem corresponding to normal skew $J$-Hamiltonian matrices.
Finally, in Section \ref{sec6}, we present and solve the inverse eigenvalue problem and the Procrustes problem associated to normal $J$-symplectic matrices, as well as a numerical example is provided.

\section{Optimization problem for normal $J$-Hamiltonian matrices}\label{sec3}

Throughout this section, for a given full-column rank matrix $X \in {\mathbb C}^{n \times m}$,
 a given diagonal matrix $D \in {\mathbb C}^{m \times m}$, and a given matrix $\widetilde{A} \in {\mathbb C}^{n \times n}$, assuming that ${\cal S} \neq \emptyset$ in (\ref{ecua}),  we treat the existence and uniqueness problem of finding $\widehat{A} \in {\cal S} \cap {\cal NJH}$ such that
\begin{equation} \label{Procurstes1}
\min_{A \in {\cal S} \cap {\cal NJH}} \|\widetilde{A} - A\|_F = \|\widetilde{A} - \widehat{A} \|_F.
\end{equation}

\begin{remark}
As it is well known, the (general) classical eigenvalue problem has not always a solution. 
However, with respect to the inverse eigenvalue problem, 
we note that diagonalizable (even singular) matrices 
give rise to a wide class
of matrices that may appear in (\ref{ecua}), which can be constructed by different matrices $X$ and $D$. For instance, ${\cal S} \neq
\emptyset$ for fixed diagonal matrices  $X,D$ (because $A=D$ provides, trivially, a solution); whence 
we can assure that infinitely many solutions are found (since all matrices $B$ such that $B(PX)=(PX)D$ belong to ${\cal S}$, for any nonsingular matrix $P$). Given that  the
set of diagonalizable matrices is dense in the set of all matrices \cite{HoJh}, we conclude
that, in practice, it is quite common 
to find situations where matrices solving (\ref{ecua}) do appear (for example, all matrices of index 1).
In addition, ${\cal S}$ also contains matrices satisfying $AA^*=A^*A$ and $(AJ)^*=AJ$,
for example, matrices having the form of the direct sum $D_1 \oplus (-D_1^*)$, for $D_1$ being a diagonal matrix of adequate size.
\end{remark}

\subsubsection*{\textit{Uniqueness of solution of Procrustes problem for ${\cal T} = {\cal NJH}$}}

Once the existence of solution of (\ref{Procurstes1}) is proven, for a non-empty set ${\cal S} \cap {\cal NJH}$,
we have that
$$
{\cal S}  \cap {\cal NJH} = A_0 + \{A \in
{\mathbb C}^{n \times n}: AX=0\}
$$
is an affine subspace of ${\mathbb C}^{n \times n}$ (with the canonical affine structure \cite{Ga}) 
for any $A_0 \in {\cal S} \cap {\cal NJH}$. By \cite[p. 245]{Ga}, the uniqueness of solution of (\ref{Procurstes1}) is guaranteed since that the distance of a point to an affine subspace is always attained in a single point.

From now on, we focus our efforts on the proof of the existence problem.

\subsubsection*{\textit{The general expression for matrices $A \in {\cal S} \cap {\cal NJH}$}}

Since $J^2=-I_n$, it is clear that $n=2k$ for some positive integer $k$.
Then, there exists a unitary matrix $U \in {\mathbb C}^{n \times n}$  such that
\begin{equation}\label{U de J}
J = U \left[\begin{array}{cc}
iI_k & O \\
O & -i I_k
\end{array}
\right] U^*.
\end{equation}
Let us consider the following partition
$$X = U \left[ \begin{array}{c}
         X_{1} \\
         X_{2}
         \end{array} \right],$$
where $X_1, X_2 \in {\mathbb C}^{k \times m}$.
The general normal $J$-Hamiltonian solution of the inverse eigenvalue problem (\ref{ecua}) is given by (see \cite[Theorem 2]{GiLeTh2})
\begin{equation}\label{solu}
A  = U \left[ \begin{array}{cc}
         A_{11} & A_{12} \\
         A_{12}^* & A_{22}
         \end{array} \right] U^*,
\end{equation}
where
\begin{equation}\label{solA12}
A_{12}= X_{1} D Q_{X_1} \left(X_{2} Q_{X_1}\right)^{\dagger} + Y_{12} P_{X_{2} Q_{X_1}},
\end{equation}
for arbitrary $Y_{12} \in {\mathbb C}^{k \times k}$,
\begin{eqnarray}\label{solA11}
A_{11} & = & \left[X_{1} D - X_{1} D Q_{X_1} \left(X_{2} Q_{X_1}\right)^{\dagger} X_{2} - Y_{12} P_{X_{2} Q_{X_1}} X_{2}\right] X_1^{\dagger} + \nonumber \\
& & + (X_1^{\dagger})^* X_2^* P_{X_{2} Q_{X_1}} Y_{12}^* P_{X_1} + P_{X_1} Y_{11} P_{X_1},
\end{eqnarray}
and
\begin{eqnarray}\label{solA22}
A_{22} & = & \left[X_{2} D - \left(Q_{X_1} X_{2}^*\right)^{\dagger} Q_{X_1} D^* X_1^* X_1 - P_{X_{2} Q_{X_1}} Y_{12}^* X_1\right] X_2^{\dagger} + \nonumber \\
& & + (X_2^{\dagger})^* \left[X_1^* X_1 D Q_{X_1}\left(X_2 Q_{X_1}\right)^{\dagger} 
- X_1^* Y_{12} P_{X_{2} Q_{X_1}} \right] P_{X_{2}} + P_{X_{2}} Y_{22} P_{X_{2}},
\end{eqnarray}
for arbitrary skew-hermitian $Y_{11}, Y_{22} \in {\mathbb C}^{k \times k}$ under the following conditions:
\begin{equation}\label{cond1}
X_{1} D Q_{X_1} Q_{X_{2} Q_{X_1}} = O,
\end{equation}
\begin{equation}\label{cond2}
\left[X_{1} D - X_{1} D Q_{X_1} (X_{2} Q_{X_1})^{\dagger} X_{2} - Y_{12} P_{X_{2} Q_{X_1}} \, X_{2}\right] Q_{X_1} = O,
\end{equation}
\begin{equation}\label{cond3}
\left[X_{2} D - (Q_{X_1} X_{2}^*)^{\dagger} Q_{X_1} D^* X_1^* X_1 - P_{X_{2} Q_{X_1}} Y_{12}^* X_1\right] Q_{X_2} = O,
\end{equation}
$A_{11} A_{12} =A_{12} A_{22}$,
and the matrices
\begin{equation}\label{skew}
X_1^* \left(X_{1} D - A_{12} X_{2}\right) \quad \text{and} \quad
X_2^* \left(X_{2} D - A_{12}^* X_1\right)
\end{equation}
are skew-hermitian.

Note that the condition (\ref{skew}) is equivalent to $X_1^* X_1 D - X_2^* X_2 D$ is skew-hermitian.

In order to find an optimal solution $\widehat{A}$, as we have stated in (\ref{Procurstes1}), we first transform the problem into a more simple one.

\subsubsection*{\textit{The simplified problem}}

Let $\widetilde{A} \in {\mathbb C}^{n \times n}$ be an arbitrary given matrix such that $U^*\widetilde{A}U$ is partitioned as
\begin{equation} \label{Atilde}
      U^*\widetilde{A}U =  \left[ \begin{array}{cc}
         \widetilde{A}_{11} & \widetilde{A}_{12} \\
         \widetilde{A}_{21} & \widetilde{A}_{22}
         \end{array} \right],
\end{equation}
where $\widetilde{A}_{ij}\in {\mathbb C}^{k \times k}$ and $U$ is the same matrix as in (\ref{U de J}).
Since Frobenius norm is unitarily invariant, by using the matrix $A \in {\cal S} \cap {\cal NJH}$ given in
(\ref{solu}), we obtain
\begin{eqnarray*}
\|\widetilde{A} -A\|^2_F & = & \|U^*\widetilde{A}U-U^*A U\|_F^2 \\
                    & = &  \|\widetilde{A}_{11}-A_{11}\|_F^2+\|\widetilde{A}_{12}-A_{12}\|_F^2+
     \|\widetilde{A}_{21}-A_{12}^*\|_F^2+\|\widetilde{A}_{22}-A_{22}\|_F^2.
\end{eqnarray*}
Now, we study each of these four norms separately.

\subsubsection*{\textit{Explicit form of the solution}}

Next results are needed in order to minimize the norm of some matrices with special structures.
We quote Lemma 1 in \cite{Tr} with a slight modification.
\begin{lemma}\label{minimo2}
Let $B \in {\mathbb C}^{q \times m}$, $P_1 \in {\mathbb C}^{q \times q}$, and
$P_2 \in {\mathbb C}^{m \times m}$ where $P_i^2=P_i=P_i^*$ for $i=1,2$.
Then the following statements hold:
\begin{enumerate}[(i)]
\item there exists the minimum of $\|B- P_1 E P_2\|_{F}$ for $E$ ranging
the set ${\mathbb C}^{q \times m}$.
\item This minimum is attained in $E$ if and only if $E$ satisfies $P_1 (B-E) P_2=O$.
\item $\displaystyle \min_{E \in {\mathbb C}^{q \times m}} \|B- P_1 E P_2\|_{F} = \|B -P_1 B P_2\|_F.$
\end{enumerate}
\end{lemma}

\begin{lemma}\label{nuestro}
Let $M \in {\mathbb C}^{m \times n}$. The equation $P_{M} Y P_{M}=O$ has always a (nontrivial) skew-hermitian solution $Y \in {\mathbb C}^{n \times m}$.
\end{lemma}
\begin{proof}
It follows as an immediate consequence of applying  \cite[Theorem 2.4]{KaMi} to $iX$
with
$C=O$ and $A=B=P_{M}$.
\end{proof}

Now, we are ready to give the explicit solution to problem (\ref{Procurstes1}).
Assume that ${\cal S} \cap {\cal NJH} \neq \emptyset$.

We start by studying $\|\widetilde{A}_{12}-A_{12}\|_F^2$ since it contains the arbitrary matrix $Y_{12}$ that also appears
in the expressions of $A_{11}$ and $A_{22}$.

Under the condition (\ref{cond1}), by using expression (\ref{solA12}) of $A_{12}$,
Lemma \ref{minimo2} assures that
$$\|\widetilde{A}_{12}-A_{12}\|_F^2=\left\|\widetilde{A}_{12} - X_{1} D Q_{X_1} \left(X_{2} Q_{X_1}\right)^{\dagger} -
Y_{12} P_{X_{2} Q_{X_1}}\right\|_F^2$$
attains its minimum in $Y_{12}$ if and only if $Y_{12}$ satisfies
$$\left[\widetilde{A}_{12}- X_{1} D Q_{X_1} \left(X_{2} Q_{X_1}\right)^{\dagger} - Y_{12}\right] P_{X_{2} Q_{X_1}}=O.$$
Since $\left(X_2 Q_{X_1}\right)^\dag P_{X_2 Q_{X_1}} = O$, the last equality simplifies to
\begin{equation}\label{12}
(\widetilde{A}_{12} - Y_{12} ) P_{X_{2} Q_{X_1}}= O.
\end{equation}
In this case,
\[
\min_{A_{12} \in {\mathbb C}^{k \times k}} \| \widetilde{A}_{12}-A_{12}\|_F^2 =
\left\| \left[\widetilde{A}_{12} X_{2} - X_1 D\right] Q_{X_1} \left(X_{2} Q_{X_1}\right)^{\dagger}\right\|_F^2,
\]
since $R-R P_{M} = R MM^\dag$ for arbitrary matrices $R$ and $M$ of adequate sizes.

Recalling that the Frobenius norm is invariant under taking adjoint and applying the last result to
$$
\|\widetilde{A}_{21}-A_{12}^* \|_F^2 = \|((\widetilde{A}_{21})^* -A_{12})^* \|_F^2 = \|(\widetilde{A}_{21})^* -A_{12} \|_F^2
$$
we have
\[
\min_{A_{12} \in {\mathbb C}^{k \times k}} \| (\widetilde{A}_{21})^*-A_{12}\|_F^2 =
\left\| \left[(\widetilde{A}_{21})^* X_{2} - X_1 D\right] Q_{X_1} \left(X_{2} Q_{X_1}\right)^{\dagger}\right\|_F^2
\]
if and only if $Y_{12}$ satisfies
\begin{equation}\label{21}
((\widetilde{A}_{21})^* - Y_{12} ) P_{X_{2} Q_{X_1}}=O.
\end{equation}
From (\ref{solA12}) we obtain
\begin{equation}\label{A12}
\widehat{A}_{12}= X_{1} D Q_{X_1} \left(X_{2} Q_{X_1}\right)^{\dagger} + \widetilde{A}_{12} P_{X_{2} Q_{X_1}}.
\end{equation}

It is clear that from (\ref{12}) and (\ref{21}) we get
$$
\widetilde{A}_{12} P_{X_{2} Q_{X_1}}= Y_{12} P_{X_{2} Q_{X_1}} = (\widetilde{A}_{21})^* P_{X_{2} Q_{X_1}}.
$$
Conversely, if $\widetilde{A}_{12} P_{X_{2} Q_{X_1}} = (\widetilde{A}_{21})^* P_{X_{2} Q_{X_1}}=:Z$,
post-multiplying both sides by $P_{X_{2} Q_{X_1}}$ we get
$$
\widetilde{A}_{12} P_{X_{2} Q_{X_1}}= (\widetilde{A}_{21})^*P_{X_{2} Q_{X_1}}= Z P_{X_{2} Q_{X_1}}.
$$
Now, $Y_{12}:=Z$ satisfies (\ref{12}) and (\ref{21}).

Since $Y_{12} P_{X_{2} Q_{X_1}} = \widetilde{A}_{12} P_{X_{2} Q_{X_1}} = (\widetilde{A}_{21})^* P_{X_{2} Q_{X_1}}$,
the conditions (\ref{cond2}) and (\ref{cond3}) are equivalent to
\begin{equation}\label{cond22}
\left[X_{1} D - X_{1} D Q_{X_1} (X_{2} Q_{X_1})^{\dagger} X_{2}- \widetilde{A}_{12}P_{X_{2} Q_{X_1}} X_{2}\right]Q_{X_1}= O
\end{equation}
 and
\begin{equation}\label{cond33}
\left[X_{2} D -(Q_{X_1} X_{2}^*)^{\dagger} Q_{X_1} D^* X_1^* X_1 - P_{X_{2} Q_{X_1}} \widetilde{A}_{21} X_1\right] Q_{X_2} = O,
\end{equation}
respectively. Moreover, expression (\ref{solA11}) gives
\begin{eqnarray*}
\| \widetilde{A}_{11}-A_{11}\|_F^2 & = &
\left\| \widetilde{A}_{11} - \left[X_{1} D - X_{1} D Q_{X_1} \left(X_{2} Q_{X_1}\right)^{\dagger} X_{2} -
\widetilde{A}_{12} P_{X_{2} Q_{X_1}} X_{2}\right] X_1^{\dagger} - \right. \\
& & \left. -(X_1^{\dagger})^* X_2^* P_{X_{2} Q_{X_1}} \widetilde{A}_{21} P_{X_1} - P_{X_1} Y_{11} P_{X_1} \right\|_F^2.
\end{eqnarray*}
Denoting by $R$ the fixed terms in the last norm, that is,
\begin{eqnarray*}
R& = & \widetilde{A}_{11} - \left[X_{1} D - X_{1} D Q_{X_1} \left(X_{2} Q_{X_1}\right)^{\dagger} X_{2} -
\widetilde{A}_{12} P_{X_{2} Q_{X_1}} X_{2}\right] X_1^{\dagger} -\\
& & -(X_1^{\dagger})^* X_2^* P_{X_{2} Q_{X_1}} \widetilde{A}_{21} P_{X_1},
\end{eqnarray*}
Lemma \ref{minimo2} ensures that $\| \widetilde{A}_{11}-A_{11}\|_F^2$ attains its minimum in $Y_{11}$ if and only if
$P_{X_1} (R-Y_{11}) P_{X_1}=O$, which is equivalent to
\begin{equation}\label{11}
P_{X_1} (\widetilde{A}_{11}-Y_{11}) P_{X_1}=O,
\end{equation}
since $X_1^{\dagger} P_{X_1}=O$ and $P_{X_1} (X_1^{\dagger})^*=O$. In this case,
$$
\min_{A_{11} \in {\mathbb C}^{k \times k}} \| \widetilde{A}_{11}-A_{11}\|_F^2 = \| R-P_{X_1} \widetilde{A}_{11} P_{X_1} \|_F^2.
$$
Then, by using (\ref{solA11}) we obtain
\begin{eqnarray}\label{A11}
\widehat{A}_{11} & = & \left[X_{1} D - X_{1} D Q_{X_1} \left(X_{2} Q_{X_1}\right)^{\dagger} \, X_{2} -
\widetilde{A}_{12} P_{X_{2} Q_{X_1}} X_{2}\right] X_1^{\dagger} + \nonumber\\
& & +(X_1^{\dagger})^* X_2^* P_{X_{2} Q_{X_1}} \widetilde{A}_{21} P_{X_1} + P_{X_1} \widetilde{A}_{11} P_{X_1}.
\end{eqnarray}
By Lemma \ref{nuestro}, (\ref{11}) has always a skew-hermitian solution $\widetilde{A}_{11}-Y_{11}$ for a skew-hermitian $Y_{11}$.
Then, $\widetilde{A}_{11}$ is skew-hermitian
as well.

Conversely, if $\widetilde{A}_{11}$ is skew-hermitian then equation $P_{X_1} C P_{X_1}=O$ has a skew-hermitian solution $C$. Setting
$Y_{11}:=\widetilde{A}_{11}-C$, we get that $Y_{11}$ is a skew-hermitian matrix that satisfies (\ref{11}).

A similar reasoning allows us to say that the minimum of $\| \widetilde{A}_{22}-A_{22}\|_F^2$ is attained in $Y_{22}$
if and only if
$P_{X_2} (\widetilde{A}_{22}-Y_{22}) P_{X_2}=O$ which is equivalent to $\widetilde{A}_{22}$ is skew-hermitian.
Finally, from (\ref{solA22}) the last norm is minimized by
\begin{eqnarray}\label{A22}
\widehat{A}_{22} & = &
\left[X_{2} D - \left(Q_{X_1)} X_{2}^*\right)^{\dagger} Q_{X_1} D^* X_1^* X_1 - P_{X_{2} Q_{X_1}} \widetilde{A}_{21} X_1\right] X_2^{\dagger} + \nonumber \\
& & + (X_2^{\dagger})^* \left[X_1^* X_1 D Q_{X_1}\left(X_2 Q_{X_1}\right)^{\dagger} -
X_1^* \widetilde{A}_{12} P_{X_{2} Q_{X_1}} \right] P_{X_{2}} + P_{X_{2}} \widetilde{A}_{22} P_{X_{2}}.
\end{eqnarray}

The results obtained can be summarized in the following theorem.

\begin{theorem}\label{gran}
Let $\widetilde{A} \in {\mathbb C}^{n \times n}$ be a partitioned matrix as in (\ref{Atilde})
and assume that ${\cal S} \cap {\cal NJH} \neq \emptyset$.
Then the problem (\ref{Procurstes1}) has a unique solution $\widehat{A} \in {\cal S} \cap {\cal NJH}$ if and only if
$\widetilde{A}_{11}$, $\widetilde{A}_{22}$, and $X_1^* X_1 D - X_2^* X_2 D$ are skew-hermitian,
(\ref{cond1}), (\ref{cond22}), (\ref{cond33}), and
$(\widetilde{A}_{12}-(\widetilde{A}_{21})^*) P_{X_{2} Q_{X_1}} = O$ are satisfied and if we define $\widehat{A}_{11}$, $\widehat{A}_{12}$, and $\widehat{A}_{22}$ by (\ref{A11}), (\ref{A12}) and (\ref{A22}), respectively, the condition $\widehat{A}_{11} \widehat{A}_{12}=\widehat{A}_{12} \widehat{A}_{22}$ must be fulfilled.
In this case, the (unique) optimal solution is given by
\begin{equation*} 
      \widehat{A} = U \left[ \begin{array}{cc}
         \widehat{A}_{11} & \widehat{A}_{12} \\
         (\widehat{A}_{12})^* & \widehat{A}_{22}
         \end{array} \right] U^*.
\end{equation*}
\end{theorem}

By using the condition given in the statement of Theorem \ref{gran}, we notice that
$$
\widehat{A}_{12} = X_{1} D Q_{X_1} \left(X_{2} Q_{X_1}\right)^{\dagger} + (\widetilde{A}_{21})^* P_{X_{2} Q_{X_1}}
$$
gives an alternative expression for $\widehat{A}_{12}$.

\begin{remark}
It is well known that the set ${\cal S}$ is wide enough for Hermitian and generalized skew-Hamiltonian matrices as it was established in \cite{Ba}, and for normal skew $J$-Hamiltonian matrices
as showed in \cite{GiLeTh2,ZhZh}, for the particular matrix $J=\left[\begin{array}{cc}
O & I_k\\
-I_k & O\end{array}\right]$. It is clear that all these cases are particular ones of normal matrices. So, it is natural that the set ${\cal S}$ for
${\cal NJH}$ must be wider. 
\end{remark}

\subsubsection*{\textit{Algorithm and numerical examples}}\label{sec4}

The algorithm presented below indicates a procedure that solves the Procrustes problem associated to the  inverse eigenvalue problem stated in this subsection.

\begin{quote}
{\bf\sc Algorithm 1}

{\sl Input}: Matrices $\widetilde{A}, J, X$, and $D$ such that $J^2=-I_n$ and $X$ and $D$ constrained to the condition ${\cal S} \cap {\cal NJH} \neq \emptyset$.

{\sl Output}: The optimal matrix $\widehat{A}$ that solves Problem (\ref{Procurstes1}).

\begin{description}

\item[Step 1] \; Diagonalize $J=U(iI_k \oplus -iI_{k})U^*$.

\item[Step 2] \; Partition  $ U^*X= \left[ \begin{array}{c}
         X_{1} \\
         X_{2}
         \end{array} \right]$.

\item[Step 3] \; Partition $U^* \widetilde{A} U = \left[ \begin{array}{cc}
         \widetilde{A}_{11} & \widetilde{A}_{12} \\
         \widetilde{A}_{21} & \widetilde{A}_{22}
         \end{array} \right]$.

\item[Step 4] \; Compute $P=I_m -X_1^\dag X_1$ and $T=I_m -X_2^\dag X_2$.

\item[Step 5] \; Compute $L=I_m - (X_2P)^\dag X_2P$ and $Q=I_k -X_2P (X_2P)^\dag$.

\item[Step 6] \; Compute $M=X_{1} D - X_{1} D P \left(X_{2} P\right)^{\dagger} \, X_{2} -
\widetilde{A}_{12} Q \, X_{2}$ and $N=X_{2} D - \left(P X_{2}^*\right)^{\dagger} P D^* X_1^* X_1 - Q \widetilde{A}_{21} X_1$.

\item[Step 7] \; If $\widetilde{A}_{11}$  is not skew-hermitian or $\widetilde{A}_{22}$ is not skew-hermitian, go to Step 18.

\item[Step 8] \; If $X_1^* X_1 D - X_2^* X_2 D$ is not skew-hermitian, go to Step 18.

\item[Step 9] \; If $(\widetilde{A}_{12}-(\widetilde{A}_{21})^*) Q \neq O$ or $X_1 D P L \neq O$, go to Step 18.

\item[Step 10] \; If $M P \neq O$ or $N T \neq O$, go to Step 18.

\item[Step 11] \; Compute $R=I_k-X_1 X_1^\dag$ and $S=I_k-X_2 X_2^\dag$.

\item[Step 12] \; Compute $\widehat{A}_{11} = M X_1^{\dagger}  +(X_1^{\dagger})^* X_2^* Q \widetilde{A}_{21} R + R \widetilde{A}_{11} R$.

\item[Step 13] \; Compute $\widehat{A}_{12}= X_{1} D P \left(X_{2} P\right)^{\dagger} + \widetilde{A}_{12} Q$.

\item[Step 14] \; Compute $\widehat{A}_{22} = N X_2^{\dagger}+(X_2^{\dagger})^* \left[X_1^* X_1 D P\left(X_2 P\right)^{\dagger} -
X_1^* \widetilde{A}_{12} Q \right] S + S \widetilde{A}_{22} S$.

\item[Step 15] \; If $\widehat{A}_{11} \widehat{A}_{12} - \widehat{A}_{12} \widehat{A}_{22} \neq O$, go to Step 18.

\item[Step 16] \; Compute $\widehat{A} = U \left[ \begin{array}{cc}
         \widehat{A}_{11} & \widehat{A}_{12} \\
         (\widehat{A}_{12})^* & \widehat{A}_{22}
         \end{array} \right] U^*$.

\item[Step 17] \; Go to End.

\item[Step 18] \; "There is not a normal $J$-Hamiltonian matrix $\widehat{A}$ such that
                  $\displaystyle \min_{A \in {\cal S}} \|\widetilde{A} - A\|_F = \|\widetilde{A} - \widehat{A} \|_F$".

\item[End]
\end{description}
\end{quote}

\begin{remark}
We have used MATLAB R2022a and MATHEMATICA 13.1 packages.
Our worked examples have been considered for matrices of small sizes by using symbolic computations.

\end{remark}
In what follows, we present some illustrative examples. 

\begin{example}\label{ex1}
Let consider the matrix
\[
J=\left[\begin{array}{ccrr}
0 & 0 & -1 & 0 \\
0 & 0 & 0 & -1 \\
1 & 0 & 0 & 0 \\
0 & 1 & 0 & 0
\end{array}\right] = U
\left[\begin{array}{ccrr}
i & 0 & 0 & 0 \\
0 & i & 0 & 0 \\
0 & 0 & -i & 0 \\
0 & 0 & 0 & -i
\end{array}
\right] U^*,
\]
where
\[
 U = \frac{1}{\sqrt{2}} \left[
\begin{array}{ccrr}
 i & 0 & -i & 0 \\
 0 & i & 0 & -i \\
 1 & 0 & 1 & 0 \\
 0 & 1 & 0 & 1 \\
\end{array}
\right].
\]
Moreover, we consider the matrices
\[
X=\left[
\begin{array}{rcr}
 -\sqrt{2} & \sqrt{2} & 0 \\
 0 & 0 & -\frac{i a}{\sqrt{2}} \\
 \sqrt{2} & \sqrt{2} & 0 \\
 0 & 0 & \frac{a}{\sqrt{2}} \\
\end{array}
\right],  \text{ with } a \in \mathbb C,  \qquad
D=\left[\begin{array}{ccr}
1+i &  0 &   0 \\
  0 & -1+i & 0 \\
  0 &  0   &  i
\end{array}\right],
\]
and infinitely many data depending on parameters $b, h, k, z \in \mathbb C$ given by
 \[
\widetilde{A} = \frac{1}{2} \left[
\begin{array}{cccc}
5 i-2 {\rm Re}(h) & -z-\bar{b}-i \sqrt{3} & 9-2{\rm Im}(h) & -i z+i \bar{b}+3 \sqrt{3} \\
-b- \bar{z}-i \sqrt{3} & 7 i-2k & 3 \sqrt{3}-i b+i \bar{z} & 3 \\
 -9 - {\rm Im}(h) & -3 \sqrt{3}- i z+i \bar{b} & 5 i+2 {\rm Re}(h) & z+ \bar{b}-i \sqrt{3} \\
 -3 \sqrt{3}-i b+i \bar{z} & -3 & b+ \bar{z}-i \sqrt{3} & 7 i +2k \\
\end{array}
\right]
\]

Applying Algorithm 1 we get
\[
X_1 = \left[
\begin{array}{ccc}
 1+i & 1-i & 0 \\
 0 & 0 & 0 \\
\end{array}
\right], \quad
X_2 = \left[
\begin{array}{ccc}
 1-i & 1+i & 0 \\
 0 & 0 & a \\
\end{array}
\right],
\]
\[
P = \frac{1}{2} \left[
\begin{array}{rcc}
 1 & i & 0 \\
 - i & 1 & 0 \\
 0 & 0 & 2
\end{array}
\right], \qquad T = \frac{1}{2}\left[
\begin{array}{crc}
 1 & - i & 0 \\
  i & 1 & 0 \\
 0 & 0 & 0
\end{array}
\right],
\]
\[
L = \frac{1}{2}\left[
\begin{array}{crc}
 1  & - i & 0 \\
  i & 1 & 0 \\
 0 & 0 & 0
\end{array}
\right], \qquad Q = \left[
\begin{array}{cc}
 0 & 0 \\
 0 & 0
\end{array}
\right],
\]
\[
M = \left[
\begin{array}{ccc}
 -1+i  & 1+ i & 0 \\
 0 & 0 & 0
\end{array}
\right], \qquad N = \left[
\begin{array}{ccc}
 1+i  & -1+ i & 0 \\
 0 & 0 & i a
\end{array}
\right],
\]
\[
R = \left[\begin{array}{cc}
 0 & 0 \\
 0 & 1
\end{array}
\right] \qquad \text{ and } \qquad
S = \left[\begin{array}{cc}
 0 & 0 \\
 0 & 0
\end{array}
\right].
\]
Step 7 is satisfied since
\[
\widetilde{A}_{11} = \left[
\begin{array}{cc}
-2 i & -2 i \sqrt{3} \\
 -2 i \sqrt{3} & 2 i
\end{array}
\right] \quad \text{ and } \quad
\widetilde{A}_{22} = \left[
\begin{array}{cc}
 7 i &  i \sqrt{3} \\
 i \sqrt{3} & 5 i
\end{array}
\right]
\]
are skew-hermitian.
Now,
\[
X_1^* X_1 D - X_2^* X_2 D = \left[
\begin{array}{ccc}
 0 & 4+4 i & 0 \\
 -4+4i & 0 & 0 \\
 0 & 0 & -i a \bar{a}
\end{array}
\right]
\]
is skew-hermitian, so Step 8 is satisfied.

Moreover, Step 9 is valid because $(\widetilde{A}_{12}-(\widetilde{A}_{21})^*) Q = O$ and $X_1 D P L = O$ where
\[
\widetilde{A}_{12} = \left[
\begin{array}{cc}
 h & \bar{b} \\
 \bar{z} & k
\end{array}
\right] \quad \text{ and } \quad \widetilde{A}_{21} = \left[
\begin{array}{rr}
 \bar{h} & z \\
 b & k
\end{array}
\right].
\]

Furthermore, $MP=O$, $NT=O$. Hence, Step 10 is satisfied.

After doing the computations of $\widehat{A}_{11}$, $\widehat{A}_{12}$, and $\widehat{A}_{22}$ we have that $\widehat{A}_{11}\widehat{A}_{12}=\widehat{A}_{12} \widehat{A}_{22}$ and
\[
\widehat{A} = \frac{1}{2} \left[
\begin{array}{rcrr}
 2i & 0 & -2 & 0 \\
 0 & 3i & 0 & -1 \\
 -2 & 0 & 2i & 0 \\
 0 & 1 & 0 & 3 i
\end{array}
\right].
\]
We can easily check that the matrix $\widehat{A}$ is normal $J$-Hamiltonian and satisfies the equation $AX=XD$.
\end{example}

\begin{example}
Let consider the matrices $J$ and $U$ given in the previous example and consider also the matrices
\[
 \widetilde{A}= \frac{1}{2}\left[
\begin{array}{cccc}
          5 i-2 {\rm Re}(h) &          -i \sqrt{3}-z+i &        9-2 {\rm Im}(h) &         3 \sqrt{3}-i z+1 \\
   -i \sqrt{3}-i-\bar{z} &                 -4+7 i &  3 \sqrt{3}+1+i \bar{z} &                       3 \\
       -9-2 {\rm Im}(h) &        -3 \sqrt{3}-i z+1 &          5i+ 2 {\rm Re}(h) &         -i \sqrt{3}+z-i \\
 -3 \sqrt{3}+1+i \bar{z} &                     -3 &   -i \sqrt{3}+i+\bar{z} &                  4+7 i
\end{array}
\right], \quad  h, z \in {\mathbb C},
\]
\[
X = \left[
\begin{array}{ccc}
-\sqrt{2} & \sqrt{2} & 0 \\
        0     & 0 & 0 \\
\sqrt{2} & \sqrt{2} & 0 \\
 0   & 0 & i \sqrt{2}
\end{array}
\right], \qquad D = \left[
\begin{array}{ccc}
 1+i & 0 & 0 \\
 0 & -1+i & 0 \\
 0 & 0 & 2 i
\end{array}
\right].
\]
Steps 7, 8, and 9 are satisfied, $M P = O$ but since $N T = \left[
\begin{array}{ccc}
 0 & 0 & 0 \\
 1-i & -1-i & 0 \\
\end{array}
\right] \neq O$, Step 10 is not satisfied. Hence, the Algorithm ensures that the problem has no solution.
\end{example}

\section{The normal skew $J$-Hamiltonian case}\label{sec5}

This section deals with the optimization problem for normal skew $J$-Hamiltonian matrices.
For a given full-column rank matrix $X \in {\mathbb C}^{n \times m}$ and a diagonal matrix $D \in {\mathbb C}^{m \times m}$, throughout this section, we denote by  ${\cal S} \cap {\cal NJSH}$ the set of
all the solutions  $A \in {\mathbb C}^{n \times n}$  of $AX=XD$ for $A$ being a
normal skew $J$-Hamiltonian matrix.

By using the same notation as in Section \ref{sec3}, the structure of the normal skew $J$-Hamiltonian matrices
is given by (see  \cite[Lemma 1]{ZhZh})
\begin{equation*}
A  = U \left[ \begin{array}{rc}
         A_{11} & A_{12} \\
         -A_{12}^* & A_{22}
         \end{array} \right] U^*,
\end{equation*}
where $A_{11}$ and $A_{22}$ are hermitian, $A_{11} A_{12} =A_{12} A_{22}$ and $U$ is a unitary matrix diagonalizing matrix $J$.

Notice that for a  given normal skew $J$-Hamiltonian matrix $A \in {\mathbb C}^{n \times n}$, it is easy to see that
the matrix $B := i A$ is also normal and, moreover, $B$ is $J$-Hamiltonian.
Furthermore, the matrix equation $AX =XD$ is equivalent to $BX = X \widetilde{D}$, where $\widetilde{D}= i D$.
Observe that, if $D$ satisfies the restriction on the spectrum of the skew $J$-Hamiltonian  matrix $A$ (that is, if $\lambda \in \sigma(A)$ then $\{\lambda,\bar{\lambda}\} \subseteq \sigma(A)$), the restriction on the spectrum of
$\widetilde{D}$ is also satisfied (that is, if $\lambda \in \sigma(B)$ then $\{\lambda,-\lambda,\bar{\lambda},-\bar{\lambda}\} \subseteq \sigma(B)$).

Then, the general form of matrices in ${\cal S} \cap {\cal NJSH}$
is obtained by applying  \cite[Theorem 2]{GiLeTh2} to $B=i A$ and $\widetilde{D}= i D$.
The following result gives necessary and sufficient conditions that ensure the existence of solutions of the inverse eigenvalue problem for matrices having the structure considered in this section.

\begin{theorem}\label{main2}
Let $X \in {\mathbb C}^{n \times m}$ be a full-column rank matrix, $D \in {\mathbb C}^{m \times m}$ be a diagonal matrix
and $J \in {\mathbb R}^{n \times n}$ be a normal matrix such that $J^2=-I_n$, where $n=2k$.
Consider the partition $X = U \left[ \begin{array}{cc}
         X_{1}^* & X_{2}^*
         \end{array} \right]^*$
where $X_1, X_2 \in {\mathbb C}^{k \times m}$.
Then there exists $A \in {\cal S} \cap {\cal NJSH}$ if and only if the conditions: $(X_{1} D - A_{12} \, X_{2}) Q_{X_1} = O$,
$$\left[X_{1} D - X_{1} D Q_{X_1} (X_{2} Q_{X_1})^{\dagger} \, X_{2} - Y_{12} P_{X_{2} Q_{X_1}} X_{2}\right] Q_{X_1} = O,$$
$$\left[X_{2} D + (Q_{X_1} X_{2}^*)^{\dagger} Q_{X_1} D^* X_1^* X_1 + P_{X_{2} Q_{X_1}} Y_{12}^* X_1\right] Q_{X_2} = O,$$
$X_1^* X_1 D - X_2^* X_2 D$ is hermitian, and $A_{11} A_{12}=A_{12}A_{22}$ hold.
In this case, the general solution is given by
\begin{equation*}
A  = U \left[ \begin{array}{rc}
         A_{11} & A_{12} \\
         -A_{12}^* & A_{22}
         \end{array} \right] U^*,
\end{equation*}
with
\[
A_{11}  =  \left[X_{1} D - X_{1} D Q_{X_1} (X_{2} Q_{X_1})^{\dagger} X_{2} - Y_{12} P_{X_{2} Q_{X_1}} X_{2}\right] X_1^{\dagger} + (X_1^{\dagger})^* X_2^* P_{X_{2} Q_{X_1}} Y_{12}^* P_{X_1} + P_{X_1} Y_{11} P_{X_1},
\]
\begin{equation*}
A_{12}= X_{1} D Q_{X_1} (X_{2} Q_{X_1})^{\dagger} + Y_{12} P_{X_{2} Q_{X_1}},
\end{equation*}
and
\begin{eqnarray*}
A_{22} & = & \left[X_{2} D + (Q_{X_1} X_{2}^*)^{\dagger} Q_{X_1} D^* X_1^* X_1 + P_{X_{2} Q_{X_1}} Y_{12}^* X_1\right] X_2^{\dagger} + \nonumber \\
& & + (X_2^{\dagger})^* \left[X_1^* X_1 D Q_{X_1}(X_2 Q_{X_1})^{\dagger} + X_1^* Y_{12} P_{X_{2} Q_{X_1}} \right]
P_{X_2} + P_{X_2} Y_{22} P_{X_2},
\end{eqnarray*}
where $Y_{11}, Y_{22} \in {\mathbb C}^{k \times k}$ are arbitrary hermitian matrices and $Y_{12} \in {\mathbb C}^{k \times k}$ is an arbitrary matrix.
\end{theorem}

Now, the Procrustes problem associated to the matrices in ${\cal S} \cap {\cal NJSH}$ is solved by applying Theorem \ref{gran} to $B=i A$ and $\widetilde{D}= i D$.
Summarizing, we establish the following theorem.

\begin{theorem}\label{gran2}
Let $\widetilde{A} \in {\mathbb C}^{n \times n}$ be a partitioned matrix as in (\ref{Atilde})
and assume that ${\cal S} \cap {\cal NJSH} \neq \emptyset$.
Then the problem stated in (\ref{Procurstes}) has a unique solution $\widehat{A} \in {\cal S} \cap {\cal NJSH}$ if and only if
the matrices $\widetilde{A}_{11}$, $\widetilde{A}_{22}$, and $X_1^* X_1 D - X_2^* X_2 D$ are hermitian,
$(X_{1} D - A_{12} \, X_{2}) Q_{X_1} = O$,
\begin{equation*}
\left[X_{1} D - X_{1} D Q_{X_1} (X_{2} Q_{X_1})^{\dagger} X_{2} - \widetilde{A}_{12} P_{X_{2} Q_{X_1}} X_{2}\right] Q_{X_1} = O,
\end{equation*}
\begin{equation*}
\left[X_{2} D +(Q_{X_1} X_{2}^*)^{\dagger} Q_{X_1} D^* X_1^* X_1 - P_{X_{2} Q_{X_1}} \widetilde{A}_{21} X_1\right] Q_{X_2} = O,
\end{equation*}
$(\widetilde{A}_{12}+(\widetilde{A}_{21})^*) P_{X_{2} Q_{X_1}} = O$ hold; and if we define
\[
\widehat{A}_{11}  =  \left[X_{1} D - X_{1} D Q_{X_1} \left(X_{2} Q_{X_1}\right)^{\dagger} X_{2} -
\widetilde{A}_{12} P_{X_{2} Q_{X_1}} X_{2}\right] X_1^{\dagger}
+(X_1^{\dagger})^* X_2^* P_{X_{2} Q_{X_1}} \widetilde{A}_{21} Q_{X_1} + P_{X_1} \widetilde{A}_{11} P_{X_1},
\]
\begin{equation*}
\widehat{A}_{12}= X_{1} D Q_{X_1} \left(X_{2} Q_{X_1}\right)^{\dagger} + \widetilde{A}_{12} P_{X_{2} Q_{X_1}},
\end{equation*}
and
\begin{eqnarray*}
\widehat{A}_{22} & = & \left[X_{2} D + \left(Q_{X_1)} X_{2}^*\right)^{\dagger} Q_{X_1} D^* X_1^* X_1 +
                     P_{X_{2} Q_{X_1}} \widetilde{A}_{21} X_1\right] X_2^{\dagger} + \nonumber \\
& & + (X_2^{\dagger})^* \left[X_1^* X_1 D Q_{X_1}\left(X_2 Q_{X_1}\right)^{\dagger} -
X_1^* \widetilde{A}_{12} P_{X_{2} Q_{X_1}} \right] P_{X_{2}} + P_{X_{2}} \widetilde{A}_{22} P_{X_{2}},
\end{eqnarray*}
then the condition $\widehat{A}_{11} \widehat{A}_{12}=\widehat{A}_{12} \widehat{A}_{22}$ must also be fulfilled.

In this case, the unique optimal solution is given by
\begin{equation*}
      \widehat{A} = U \left[ \begin{array}{cc}
         \widehat{A}_{11} & \widehat{A}_{12} \\
         -(\widehat{A}_{12})^* & \widehat{A}_{22}
         \end{array} \right] U^*.
\end{equation*}
\end{theorem}

\subsubsection*{\textit{Algorithm and numerical examples}}

This subsection presents an algorithm to solve the Procrustes problem related to the inverse eigenvalue problem stated. For this case, it reads  more simplified. We also present some numerical examples.

\begin{quote}
{\bf\sc Algorithm 2}

{\sl Input}: Matrices $\widetilde{A}, J, X$, and $D$ such that $J^2=-I_n$ and $X$ and $D$ constrained to the condition ${\cal S} \cap {\cal NJSH} \neq \emptyset$.

{\sl Output}: The optimal matrix $\widehat{A}$ that solves Problem (\ref{Procurstes}) for ${\cal T} =  {\cal NJSH}$.

\begin{description}

\item[Step 1] \; Compute $\widetilde{D}=i D$.

\item[Step 2] \; Apply Algorithm 1 to $\widetilde{A}, J, X$, and $\widetilde{D}$ obtaining the normal $J$-Hamiltonian $\widehat{B}$.

\item[Step 3] \; The optimal solution is $\widehat{A}=-i\widehat{B}$ provided that ${\cal S} \cap {\cal NJSH} \neq \emptyset$

\item[End]
\end{description}
\end{quote}

Now, we provide some numerical examples.

\begin{example}\label{ex3}
Let consider the same matrices $J$, $U$, and $X$ as in Example \ref{ex1}.
Moreover, we consider the matrices
\[
D=\left[\begin{array}{ccr}
1+i &  0 &   0 \\
  0 & 1-i & 0 \\
  0 & 0 & 1
\end{array}\right], \qquad
\widetilde{A} = \frac{1}{4} \left[\begin{array}{rccc}
-2h+7-2 i &        1-2 i & -2-3 i-2 i h &          2-i \\
       -3-2 i &           -7 &       -2+3 i &           -i \\
 -2+7 i-2 i h &          2+i &   2 h+11+2 i &       1+10 i \\
       -2+5 i &            i &        5-6 i &           17
\end{array}\right],
\]
with $h \in \mathbb C$. Since matrix $X$ remains unchanged, by applying Algorithm 2  we get the same values for the matrices
$X_1$, $X_2$, $P$, $T$, $L$, $Q$, $R$, and $S$. Now, we have
\[
M = \left[
\begin{array}{ccc}
 1+i  & 1- i & 0 \\
 0 & 0 & 0
\end{array}
\right], \qquad N = \left[
\begin{array}{ccc}
 1-i  & 1+ i & 0 \\
 0 & 0 & a
\end{array}
\right].
\]
Steps 7 and 9 are satisfied since
\[
\widetilde{A}_{11} = \left[
\begin{array}{cr}
 1 & i \\
 -i & 1
\end{array}
\right] \quad \text{ and } \quad
\widetilde{A}_{22} = \frac{1}{4} \left[
\begin{array}{cc}
 11 + 2 \sqrt{3} & -2 + \sqrt{3} + 4i \\
  -2 + \sqrt{3} - 4i & 9-2 \sqrt{3}
\end{array}
\right]
\]
are hermitian and, moreover, $(\widetilde{A}_{12}+(\widetilde{A}_{21})^*) Q = O$ and $X_1 D P L = O$ where
\[
\widetilde{A}_{12} = \left[
\begin{array}{cc}
 1 + i & i \\
 2 & 3
\end{array}
\right] \quad \text{ and } \quad \widetilde{A}_{21} = \left[
\begin{array}{rr}
 h & 2i \\
 -i & 3
\end{array}
\right].
\]
Furthermore, $MP=O$, $NT=O$, and
\[
X_1^* X_1 D - X_2^* X_2 D = \left[
\begin{array}{ccr}
 0 & -4-4i & 0 \\
 -4+4i & 0 & 0 \\
0 & 0 & -a \bar{a}
\end{array}
\right]
\]
is hermitian. Thus, Steps 8 and 10 are satisfied.

After doing the computations of $\widehat{A}_{11}$, $\widehat{A}_{12}$, and $\widehat{A}_{22}$ we have that $\widehat{A}_{11} \widehat{A}_{12} = \widehat{A}_{12} \widehat{A}_{22}$
and
\[
\widehat{A} = \left[
\begin{array}{rcrc}
 1 & 0 & -i & 0 \\
 0 & 1 & 0 & 0 \\
 -i & 0 & 1 & 0 \\
 0 & 0 & 0 & 1
\end{array}
\right].
\]
We can easily check that the matrix $\widehat{A}$ is normal skew $J$-Hamiltonian and satisfies the equation $AX=XD$.
Furthermore, the value of the norm is
$$\|\widetilde{A} - \widehat{A} \|_F = \sqrt{|h|^2-2 {\rm Re}(h) +45},$$
and its minimum value is attained for $h = 1$ and equals $2\sqrt{11} =6.63325$.
\end{example}

\begin{example}
Let consider the matrices $J$, $U$, and $\widetilde{A}$ given in Example \ref{ex3} and also consider the matrices
\[
X=\left[
\begin{array}{rcr}
 -\sqrt{2} & \sqrt{2} & 0 \\
 0 & 0 & 0 \\
 \sqrt{2} & \sqrt{2} & 0 \\
 0 & 0 & i \sqrt{2} \\
\end{array}
\right] \quad \text{and} \quad D = \left[
\begin{array}{ccc}
 1+i & 0 & 0 \\
 0 & 1-i & 0 \\
 0 & 0 & 2
\end{array}
\right].
\]
Step 7 is satisfied, $X_1 D P L = O$ but since $(\widetilde{A}_{12}+(\widetilde{A}_{21})^*) Q =
\left[
\begin{array}{cc}
 0 & 2i \\
 0 & 6
\end{array}
\right] \neq O$, Step 9 is not satisfied. Hence, Algorithm 2 ensures that the Procrustes problem has no solution.
\end{example}


\section{The normal $J$-symplectic case}\label{sec6}

This section deals with the inverse eigenvalue problem for normal $J$-symplectic matrices and the associated optimization problem.
In order to solve them, we need to recall some results about the Cayley transform.

Let $A \in {\mathbb C}^{n \times n}$ such that $\det(I_n + A) \neq 0$, that is, $-1$ is not an eigenvalue of $A$. The Cayley transform of $A$ is the matrix given by \cite{FaTs}
\begin{equation}\label{Cayley}
A^C := \left(I_n - A\right)\left(I_n +A\right)^{-1}.
\end{equation}

We recall some useful properties:
\begin{enumerate}[(i)]
  \item if $A^C$ is the Cayley transform of some matrix $A \in {\mathbb C}^{n \times n}$, then the matrix $I_n + A^C$ is nonsingular.
  \item $\left(I_n - A \right)\left(I_n +A \right)^{-1} = \left(I_n +A \right)^{-1} \left(I_n - A \right)$.
  \item 	$A^C$ satisfies (\ref{Cayley}) if and only if $A = \left(I_n - A^C\right)\left(I_n +A^C\right)^{-1}$.
  \item The Cayley transform, $A \mapsto A^C$, is an involution on the set of matrices $A \in {\mathbb C}^{n \times n}$ such that $I_n + A$ is nonsingular, that is, $M=A^C $ if and only if $A=M^C$.
  \item $A$ and $A^C$ commute.
	\item $(A^C)^* = (A^*)^C$.
\end{enumerate}

In this section, we will consider a fixed normal matrix $J \in {\mathbb R}^{2n \times 2n}$ such that $J^2=-I_{2n}$.

\begin{lemma}\label{simp}
Let $A \in {\mathbb C}^{2n \times 2n}$ such that $\det(I_{2n} + A) \neq 0$. Then $A$ is normal and $J$-symplectic if and only if $A^C$ is normal and $J$-Hamiltonian.
\end{lemma}
\begin{proof}
It is easy to prove that $A$ is normal if and only if $(I_{2n} + A)^{-1}$ is normal.
By using (\ref{Cayley}), we get the equality $(I_{2n} + A^C)(I_{2n} + A)=2 I_{2n}$. Then, $A$ normal is equivalent to $I_{2n} + A^C$ is normal. And, finally, $A$ normal if and only if $A^C$ is normal.

Now,
for a given $J$-symplectic $A \in {\mathbb C}^{2n \times 2n}$, substituting $A= \left(I_{2n} - A^C\right)\left(I_{2n} +A^C\right)^{-1}$ in $A^* JA=J$,
we get
$$\left(\left(I_{2n} +A^C\right)^*\right)^{-1} \left(I_{2n} - A^C\right)^* J \left(I_{2n} - A^C\right)\left(I_{2n} + A^C\right)^{-1}= J.$$
Premultiplying by $\left(I_{2n} + A^C\right)^*$ and postmultiplying by $I_{2n} + A^C$ we have
$$\left(I_{2n} - A^C\right)^* J \left(I_{2n} - A^C\right)= (I_{2n} + A^C)^* J \left(I_{2n} + A^C\right).$$
Operating and using that $J^{-1} = J^* = -J$, we arrive at $J A^C = \left(J A^C\right)^*$, that is, $A^C$ is $J$-Hamiltonian.

On the other hand, let $A^C$ be $J$-Hamiltonian, that is, $JA^C=\left(JA^C \right)^*$. Since $A=\left(I_{2n} - A^C\right)\left(I_{2n} +A^C\right)^{-1}$,
and by using that $J^{-1} = J^* = -J$, we have
\begin{eqnarray*}
A^* J A & = &
\left(\left(I_{2n} +A^C\right)^{-1}\right)^{*} \left(I_{2n} - A^C\right)^* J \left(I_{2n} - A^C\right)\left(I_{2n} + A^C\right)^{-1} \\ &=& - \left(\left(I_{2n} +A^C\right)^*\right)^{-1} \left(I_{2n} - A^C\right)^* J^* \left(I_{2n} - A^C\right)\left(I_{2n} + A^C\right)^{-1} \\
&=& - \left(\left(I_{2n} +A^C\right)^*\right)^{-1} \left(J^* - (JA^C)^*\right) \left(I_{2n} - A^C\right)\left(I_{2n} + A^C\right)^{-1}  \\
&=& - \left(\left(I_{2n} +A^C\right)^*\right)^{-1} \left(-J - JA^C\right) \left(I_{2n} - A^C\right)\left(I_{2n} + A^C\right)^{-1}  \\
 &=& \left(\left(I_{2n} +A^C\right)^*\right)^{-1} J \left(I_{2n} + A^C\right) \left(I_{2n} - A^C\right)\left(I_{2n} + A^C\right)^{-1}.
\end{eqnarray*}
Now, by using property (ii), we get
\begin{eqnarray*}
A^* J A & = & \left(\left(I_{2n} +A^C\right)^{*}\right)^{-1} J \left(I_{2n} - A^C\right) =
 \left(\left(I_{2n} +A^C\right)^{*}\right)^{-1} \left(J - JA^C\right) \\
& = &  \left(\left(I_{2n} +A^C\right)^{*}\right)^{-1} \left(-J^* - (JA^C)^*\right) =
- \left(\left(I_{2n} +A^C\right)^{*}\right)^{-1} \left(J + JA^C\right)^* \\
& = & - \left(\left(I_{2n} +A^C\right)^{*}\right)^{-1} \left(I_{2n} + A^C\right)^* J^* = -J^* = J.
\end{eqnarray*}

\end{proof}

For a given normal $J$-symplectic matrix $A \in {\mathbb C}^{2n \times 2n}$ such that $\det(I_{2n} + A)\neq 0$, we have
that matrix $A^C$ is also normal and $J$-Hamiltonian. There exists a unitary matrix $U \in {\mathbb C}^{2n \times 2n}$ such that
\begin{equation}\label{J}
J = U \left[\begin{array}{cc}
iI_n & O \\
O & -i I_n
\end{array}
\right] U^*.
\end{equation}

For that matrix $U$, we consider the partition of matrix $A^C$ as follows
\[
A^C = U \left[ \begin{array}{cc}
         A_{11} & A_{12} \\
         A_{21} & A_{22}
         \end{array} \right] U^*.
\]
Applying \cite[Theorem 1]{GiLeTh2} to matrix $A^C$ and using that $A=\left(I_{2n} - A^C\right)\left(I_{2n} +A^C\right)^{-1}$, the general expression of $A$ is given by
\begin{equation*}
A  = U \left[ \begin{array}{cc}
        I_n - A_{11} & - A_{12} \\
         - A_{12}^* & I_n - A_{22}
         \end{array} \right] \left[ \begin{array}{cc}
        I_n + A_{11} & A_{12} \\
         A_{12}^* & I_n + A_{22}
         \end{array} \right]^{-1} U^*,
\end{equation*}
where $A_{11}^* = - A_{11}$, $A_{22}^* = - A_{22}$, and $A_{11}A_{12} = A_{12}A_{22}$.

Furthermore, for a given full-column rank matrix $X \in {\mathbb C}^{2n \times 2m}$ and a diagonal matrix $D \in {\mathbb C}^{2m \times 2m}$,
throughout this section, we denote
by ${\cal S} \cap {\cal NJS}$ the set of
all the solutions $A \in {\mathbb C}^{2n \times 2n}$ of $AX=XD$ for $A$ being normal $J$-symplectic such
that $\det(I_{2n} + A)\neq 0$.
By substituting matrix $A$ by $\left(I_{2n} - A^C\right)\left(I_{2n} +A^C \right)^{-1}$ in the matrix equation $AX =XD$, we get $X(I_{2m} -D) = A^C X (I_{2m} +D)$. Since $\det(I_{2n} + A)\neq 0$, $-1$ is not an eigenvalue of $A$, so $I_{2m} +D$ is nonsingular. Thus
$$AX =XD \ \Longleftrightarrow \ A^C X = X D^C.$$
As observed in the previous case, if $D$ satisfies the restriction on the spectrum of the $J$-symplectic matrix $A$,  the restriction on the spectrum of $D^C=\left(I_{2m} - D\right)\left(I_{2m} +D \right)^{-1}$ is also satisfied, that is, if $\lambda \in \sigma(A)$ such that
$\lambda \neq -1$ then $\displaystyle \left\{\frac{1-\lambda}{1+\lambda}, -\frac{1-\bar \lambda}{1+\bar \lambda}\right\} \in \sigma(A^C)$.

Then, the general form of normal $J$-symplectic matrices (such that $-1$ is not an eigenvalue), solution of the inverse eigenvalue problem given by (\ref{ecua}), is obtained by applying \cite[Theorem 2]{GiLeTh2} to $A^C$ and $D^C$.
The necessary and sufficient conditions that assure the existence of solutions of the inverse eigenvalue problem are given in the following result.

\begin{theorem}\label{main3}
Let $X \in {\mathbb C}^{2n \times 2m}$ be a full-column rank matrix, $D \in {\mathbb C}^{2m \times 2m}$ be a diagonal matrix and
$J \in {\mathbb R}^{2n \times 2n}$ be a normal matrix such that $J^2=-I_{2n}$.
Consider the partition $X = U \left[ \begin{array}{cc}
         X_{1}^* & X_{2}^*
         \end{array} \right]^*$
where $X_1, X_2 \in {\mathbb C}^{n \times 2m}$.
Then there exists a normal $J$-symplectic matrix $A$, solution of the problem (\ref{ecua}), if and only if the conditions: $X_{1} D^C Q_{X_1} Q_{X_{2} Q_{X_1}} = O$,
$$\left[X_{1}D^C - X_{1} D^C Q_{X_1} (X_{2} Q_{X_1})^{\dagger} \, X_{2} - Y_{12} P_{X_{2} Q_{X_1}} X_{2}\right] Q_{X_1} = O,$$
$$\left[X_{2} D^C - (Q_{X_1} X_{2}^*)^{\dagger} Q_{X_1} (D^C)^* X_1^* X_1 - P_{X_{2} Q_{X_1}} Y_{12}^* X_1\right] Q_{X_2} = O,$$
$X_1^* X_1 D^C - X_2^* X_2 D^C$ is skew-hermitian, and $A_{11} A_{12}=A_{12}A_{22}$ hold.
In this case, the general solution is given by
\begin{equation*}
A = U \left[ \begin{array}{cc}
        I_n - A_{11} & -A_{12} \\
         -A_{12}^* & I_n - A_{22}
         \end{array} \right] \left[ \begin{array}{cc}
        I_n + A_{11} & A_{12} \\
         A_{12}^* & I_n + A_{22}
         \end{array} \right]^{-1} U^*,
\end{equation*}
with
\begin{eqnarray*}
A_{11} & = & \left[X_{1}D^C - X_{1} D^C Q_{X_1} (X_{2} Q_{X_1})^{\dagger} X_{2} - Y_{12} P_{X_{2} Q_{X_1}} X_{2}\right] X_1^{\dagger} + (X_1^{\dagger})^* X_2^* P_{X_{2} Q_{X_1}} Y_{12}^* P_{X_1} + \nonumber \\
& & + P_{X_1} Y_{11} P_{X_1},
\end{eqnarray*}
\begin{equation*}
A_{12}= X_{1} D^C Q_{X_1} (X_{2} Q_{X_1})^{\dagger} + Y_{12} P_{X_{2} Q_{X_1}},
\end{equation*}
and
\begin{eqnarray*}
A_{22} & = & \left[X_{2} D^C- (Q_{X_1} X_{2}^*)^{\dagger} Q_{X_1} (D^C)^* X_1^* X_1 -P_{X_{2} Q_{X_1}} Y_{12}^* X_1\right] X_2^{\dagger} + \nonumber \\
& & + (X_2^{\dagger})^* \left[X_1^* X_1 D^C Q_{X_1}(X_2 Q_{X_1})^{\dagger} - X_1^* Y_{12} P_{X_{2} Q_{X_1}} \right]
P_{X_2} + P_{X_2} Y_{22} P_{X_2},
\end{eqnarray*}
where $Y_{11}, Y_{22} \in {\mathbb C}^{n \times n}$ are arbitrary skew-hermitian matrices and $Y_{12} \in {\mathbb C}^{n \times n}$
is an arbitrary matrix.
\end{theorem}

Now, let $\widetilde{A} \in {\mathbb C}^{2n \times 2n}$ be an arbitrary matrix such that $U^*\widetilde{A}U$ is partitioned as
\begin{equation} \label{Atil}
      U^*\widetilde{A}U =  \left[ \begin{array}{cc}
         \widetilde{A}_{11} & \widetilde{A}_{12} \\
         \widetilde{A}_{21} & \widetilde{A}_{22}
         \end{array} \right],
\end{equation}
where $\widetilde{A}_{ij}\in {\mathbb C}^{n \times n}$ and $U$ is the same matrix as in (\ref{J}).
Assuming that ${\cal S} \cap {\cal NJH} \neq \emptyset$ and applying Theorem \ref{gran} to $A^C$ and $D^C$, the Procrustes problem associated to the
matrices in ${\cal S} \cap {\cal NJS}$ can be completely solved.
The results are summarized in the following theorem.

\begin{theorem}\label{gran3}
Let $\widetilde{A} \in {\mathbb C}^{2n \times 2n}$ be a partitioned matrix as in (\ref{Atil})
and assume that ${\cal S} \cap {\cal NJS} \neq \emptyset$.
There exists a unique solution $\widehat{A}$ normal and $J$-symplectic of (\ref{Procurstes}) that also satisfies problem (\ref{ecua}) if and only if
the matrices $\widetilde{A}_{11}$, $\widetilde{A}_{22}$, and $X_1^* X_1 D^C - X_2^* X_2 D^C$ are skew-hermitian,
and the following conditions:
$X_{1} D^C Q_{X_1} Q_{X_{2} Q_{X_1}} = O$,
\begin{equation*}
\left[X_{1} D^C - X_{1} D^C Q_{X_1} (X_{2} Q_{X_1})^{\dagger} X_{2} - \widetilde{A}_{12} P_{X_{2} Q_{X_1}} X_{2}\right] Q_{X_1} = O,
\end{equation*}
\begin{equation*}
\left[X_{2} D^C-(Q_{X_1} X_{2}^*)^{\dagger} Q_{X_1} (D^C)^* X_1^* X_1 - P_{X_{2} Q_{X_1}} \widetilde{A}_{21} X_1\right] Q_{X_2} = O,
\end{equation*}
$(\widetilde{A}_{12}-(\widetilde{A}_{21})^*) P_{X_{2} Q_{X_1}} = O$ hold; and if we define
\begin{eqnarray*}
\widehat{A}_{11} & = & \left[X_{1} D^C - X_{1} D^C Q_{X_1} \left(X_{2} Q_{X_1}\right)^{\dagger} X_{2} -
\widetilde{A}_{12} P_{X_{2} Q_{X_1}} X_{2}\right] X_1^{\dagger} + \nonumber\\
& & +(X_1^{\dagger})^* X_2^* P_{X_{2} Q_{X_1}} \widetilde{A}_{21} P_{X_1} + P_{X_1} \widetilde{A}_{11} P_{X_1},
\end{eqnarray*}
\begin{eqnarray*}
\widehat{A}_{22} & = & \left[X_{2} D^C - \left(Q_{X_1} X_{2}^*\right)^{\dagger} Q_{X_1} (D^C)^* X_1^* X_1 -
                     P_{X_{2} Q_{X_1}} \widetilde{A}_{21} X_1\right] X_2^{\dagger} + \nonumber \\
& & + (X_2^{\dagger})^* \left[X_1^* X_1 D^C Q_{X_1}\left(X_2 Q_{X_1}\right)^{\dagger} -
X_1^* \widetilde{A}_{12} P_{X_{2} Q_{X_1}} \right] P_{X_{2}} + P_{X_{2}} \widetilde{A}_{22} P_{X_{2}},
\end{eqnarray*}
and \begin{equation*}
\widehat{A}_{12}= X_{1} D^C Q_{X_1} \left(X_{2} Q_{X_1}\right)^{\dagger} + \widetilde{A}_{12} P_{X_{2} Q_{X_1}},
\end{equation*}
then the condition $\widehat{A}_{11} \widehat{A}_{12}=\widehat{A}_{12} \widehat{A}_{22}$ must be fulfilled.
In this case, the unique optimal solution is given by
\begin{equation*} 
      \widehat{A} = U \left[ \begin{array}{cc}
         I_n - \widehat{A}_{11} & -\widehat{A}_{12} \\
         -\widehat{A}_{12}^* & I_n - \widehat{A}_{22}
         \end{array} \right] \left[ \begin{array}{cc}
         I_n + \widehat{A}_{11} & \widehat{A}_{12} \\
         \widehat{A}_{12}^* & I_n + \widehat{A}_{22}
         \end{array} \right]^{-1} U^*.
\end{equation*}
\end{theorem}

\subsubsection*{\textit{Algorithm and numerical example}}

The following  algorithm  allows to solve the Procrustes problem related to the inverse eigenvalue problem for normal $J$-symplectic matrices.

\begin{quote}
{\bf\sc Algorithm 3}

{\sl Input}: Matrices $\widetilde{A}, J, X$, and $D$ such that $J^2=-I_{2n}$ and $X$ and $D$ constrained to the condition ${\cal S} \cap {\cal NJS} \neq \emptyset$.

{\sl Output}: The optimal matrix $\widehat{A}$ that solves Problem (\ref{Procurstes}) for ${\cal T} =  {\cal NJS}$.

\begin{description}

\item[Step 1] \; Compute $D^C=(I_{2m}-D)(I_{2m}+D)^{-1}$.

\item[Step 2] \; Apply Algorithm 1 to $\widetilde{A}, J, X$, and $D^C$ obtaining the normal $J$-Hamiltonian matrix $\widehat{B}$.

\item[Step 3] \; The optimal solution is $\widehat{A}=  (I_{2n} - \widehat{B})(I_{2n} + \widehat{B})^{-1}$.
\item[End]

\end{description}
\end{quote}

\begin{example}\label{ex}
Let consider the matrix
\[
J=\left[\begin{array}{ccrr}
0 & 0 & -1 & 0 \\
0 & 0 & 0 & -1 \\
1 & 0 & 0 & 0 \\
0 & 1 & 0 & 0
\end{array}\right] = U
\left[\begin{array}{ccrr}
i & 0 & 0 & 0 \\
0 & i & 0 & 0 \\
0 & 0 & -i & 0 \\
0 & 0 & 0 & -i
\end{array}
\right] U^*,
\]
where
\[
 U = \frac{1}{\sqrt{2}} \left[
\begin{array}{ccrr}
 0 & i & -i & 0 \\
 i & 0 & 0 & -i \\
 0 & 1 & 1 & 0 \\
 1 & 0 & 0 & 1 \\
\end{array}
\right].
\]
Moreover, we consider the matrices
\[
X=\frac{1}{\sqrt{2}}\left[
\begin{array}{rcrc}
0 &i a & 0 & i g  \\
0 &i a & 0 &  -i g  \\
0 & -a & 0 &  g  \\
0 & a & 0 &  g  \\
\end{array}
\right],  \text{ with } a,g \in \mathbb C,  \quad
D=\left[\begin{array}{cccc}
1-i &  0 &   0  & 0\\
  0 & 1+i & 0 & 0\\
  0 &  0   &  \frac{1}{2}-\frac{i}{2} & 0\\
0 & 0 & 0 & \frac{1}{2}+\frac{i}{2}
\end{array}\right],
\]
\[
 \text{and} \quad \widetilde{A} =
\left[ \begin{array}{cccc}
3 i &          -\frac{1}{5}-i &          3  &        3  \\
     -\frac{1}{5}-i &                3 i &  3  &                        3 \\
         -3  &       -3  &            3 i &          \frac{1}{5}-i  \\
 -3  &                       -3 &      \frac{1}{5}-i    &                 3 i
\end{array}\right].
\]
It can be easily checked that
\[
D^C=\dfrac{1}{5}\left[\begin{array}{cccc}
-1+2i &  0 &   0  & 0\\
  0 & -1-2i & 0 & 0\\
  0 &  0   &  1+2i & 0\\
0 & 0 & 0 & 1-2i
\end{array}\right].
\]
Applying Algorithm 1 to $\widetilde{A}, J, X$, and $D^C$ we get
\[
X_1 = \left[
\begin{array}{cccc}
 0 & a & 0 & 0 \\
 0 & 0 & 0 & g\\
\end{array}
\right], \quad
X_2 = \left[
\begin{array}{cccc}
 0 & -a & 0 & 0 \\
 0 & 0 & 0 & g \\
\end{array}
\right],
\]
\[
P = \left[
\begin{array}{cccc}
 1 & 0 & 0 & 0\\
 0 & 0 & 0 & 0\\
 0 & 0 & 1 & 0\\
 0 & 0 & 0 & 0
\end{array}
\right], \quad T = \left[
\begin{array}{cccc}
 1 & 0 & 0 & 0\\
 0 & 0 & 0 & 0\\
 0 & 0 & 1 & 0\\
 0 & 0 & 0 & 0
\end{array}
\right],
\]
\[
L = \left[
\begin{array}{cccc}
 1 & 0 & 0 & 0\\
 0 & 1 & 0 & 0\\
 0 & 0 & 1 & 0\\
 0 & 0 & 0 & 1
\end{array}
\right], \quad Q = \left[
\begin{array}{cc}
 1 & 0 \\
 0 & 1
\end{array}
\right], \quad M =\frac{2}{5} \left[
\begin{array}{cccc}
 0 & - i a & 0 & 0 \\
 0 & 0 & 0 & -i g
\end{array}
\right],
\]
\[
 N = \frac{2}{5} \left[
\begin{array}{cccc}
 0 & i a & 0 & 0 \\
 0 & 0 & 0 & -i g
\end{array}
\right],
\quad
R = \left[
\begin{array}{cc}
 0 &  0 \\
 0 & 0
\end{array}
\right] \quad  \;   and  \;  \quad S = \left[
\begin{array}{cc}
 0 &  0 \\
 0 & 0
\end{array}
\right].
\]

\noindent Step 7 of Algorithm 1 is satisfied since
\[
\widetilde{A}_{11} = \left[
\begin{array}{cr}
 0 & -4 i \\
 -4 i & 0
\end{array}
\right] \quad \text{ and } \quad
\widetilde{A}_{22} =  \left[
\begin{array}{cc}
 6 i & 2 i \\
  2 i  & 6 i
\end{array}
\right]
\]
are both skew-hermitian.

\noindent Since $X_1^* X_1 D^C - X_2^* X_2 D^C=O$, this matrix is skew-hermitian, so Step 8 is satisfied.

\noindent As
\[
\widetilde{A}_{12} = \widetilde{A}_{21} = \frac{1}{5} \left[
\begin{array}{cc}
 1  &  0 \\
0  & 1
\end{array}
\right]
\]
Step 9 is fullfilled because $(\widetilde{A}_{12}-(\widetilde{A}_{21})^*) Q= O$ and $X_1 D^C P L = O$.

\noindent Moreover, $MP=O$ and $NT=O$, so Step 10 is satisfied.

\noindent Carrying out the calculation of $\widehat{A}_{11}$,  $\widehat{A}_{12}$, and $\widehat{A}_{22}$ we have that $\widehat{A}_{11} \widehat{A}_{12} = \widehat{A}_{12} \widehat{A}_{22}$ and
\[
\widehat{B} = \dfrac{1}{5}\left[
\begin{array}{cccc}
-2 i & -1 & 0 &0\\
-1 & -2 i &  0 & 0 \\
0 & 0 & -2 i & i \\
0 & 0 & 1 & -2 i
\end{array}
\right].
\]
\noindent This matrix $\widehat{B}$ is normal $J$-Hamiltonian and satisfies the equation $AX = XD^C$.

\noindent Finally, the optimal solution is
\[
\widehat{A} = \dfrac{1}{4}\left[
\begin{array}{cccc}
3+ 3 i & 1+ i& 0 &0\\
1+i & 3+ 3 i &  0 & 0 \\
0 & 0 & 3+ 3 i & -1- i \\
0 & 0 & -1-i & 3+3i
\end{array}
\right]
\]
\noindent which is normal $J$-symplectic and satisfies the equation $AX=XD$.
\end{example}

The matrix equation for an undamped mass-spring system with $n$  degrees of freedom looks like
\[
M \ddot{X}+K X=O,
\]
where $Y$ is a time-dependent vector that describes the motion, $M=\textnormal{diag}(m_1,m_2,\cdots,m_n)$ is the mass matrix and \[K=\left[
\begin{array}{ccccccc}
k_1+k_2 & -k_2 & 0& \cdots& 0 & 0 & 0\\
-k_2 & k_2+ k_3 & -k_3 & \dots & 0 & 0 & 0\\
0 & - k_3 & k_3+k_4 & \dots & 0& 0&0 \\
\vdots  & \vdots & \vdots &  \ddots& \vdots & \vdots & \vdots \\
0  & 0 &0 & \cdots & k_{n-2}+k_{n-1} &  -k_{n-1} & 0\\
0  & 0 &0 & \cdots & -k_{n-1} &  k_{n-1}+k_n & -k_n\\
0  & 0  &0 &  \cdots& 0 &  -k_n &  k_n\\
\end{array}
\right]\] is the stiffness matrix which is tridiagonal. Both are symmetric and  the mass and stiffness parameters must be positive.

Since we seek harmonic solutions for $Y$, it has the form $Y=[y_1 \ y_2 \ \cdots \ y_n]$ where $y_j= b_j \sin (\omega t)$. Substituting into the motion equation we get \[M (-\omega^2) Y+K Y=0\]

Given that the matrix $M$ is nonsingular, we can rewrite the last equation as
\[
M^{-1} K Y=\omega^2 Y.
\]
This problem can be regarded as an eigenvector and eigenvalue problem where the unknown matrix to find is $K$.

Let $X$ be the matrix whose columns are the eigenvectors and $D$ the diagonal matrix whose elements are the eigenvalue. In this case, the general matrix equation becomes $AX = X D$ where $A=M^{-1} K$.

Let consider the input matrices
\[M=\left[
\begin{array}{cc}
\sqrt{2} & 0\\
0 &   \sqrt{2}
\end{array}
\right], \; X=\left[
\begin{array}{cc}
1/\sqrt{2} & -1/\sqrt{2}\\
1/\sqrt{2} & 1/\sqrt{2}
\end{array}
\right], \; D= \frac{1}{2}\left[
\begin{array}{cc}
 \sqrt{6}-\sqrt{2} &  0\\
 0 &   \sqrt{6}+\sqrt{2}
\end{array}
\right], \; J=\left[
\begin{array}{cc}
 0 &  1\\
 -1 &  0
\end{array}
\right].
\]
The eigenvalues of $J$ are $i$ and $-i$, $U= \dfrac{1}{\sqrt{2}}\left[
\begin{array}{cc}
1 &  -i\\
 i &   -1\end{array}
\right]$, and
\[
A= \dfrac{1}{2}\left[
\begin{array}{cc}
\sqrt{6} &  -\sqrt{2}\\
 -\sqrt{2} &   \sqrt{6}\end{array}
\right].
\]
Then, $K=M A=\dfrac{1}{2}\left[
\begin{array}{cc}
\sqrt{3} &  0\\
 0 &   \sqrt{3}\end{array}
\right].
$
\section*{Acknowledgements}

The first and third authors were partially supported by Universidad de Buenos Aires, Argentina (EXP-UBA: 13.019/2017, 20020170100350BA).
The third author is partially supported by Ministerio de Ciencia e Innovaci\'on of Spain (Grant Red de Excelencia RED2022-134176-T).


\end{document}